\documentclass[intlimits]{amsart}
\pdfoutput=1 

\usepackage[all,hyperref,numberbysection]{bi-discrete} 
\usepackage[backend=biber,maxbibnames=5,maxalphanames=5,style=alphabetic,bibencoding=utf8,giveninits,url=false,isbn=false]{biblatex}
\AtBeginBibliography{\small}

\usepackage{tikz}
\usetikzlibrary{calc}
\usetikzlibrary{positioning}

\usepackage{subcaption}

\usepackage{graphicx}

\usepackage{tikz}
\usetikzlibrary{calc}
\usepackage{pgfplots}
\pgfplotsset{
  width=.65\linewidth,
  axis background/.style={fill=black!5!white},
  grid style={densely dotted,semithick},
  legend style={
    legend columns=1,
    legend pos=outer north east
  },
  compat=1.16,
}
\pgfplotscreateplotcyclelist{MyCol1}{%
    {red,mark = *,every mark/.append style={solid,scale=0.4,fill=red}},
	{dashed,blue,mark = square,every mark/.append style={solid,scale=0.4,fill=blue}},
    {dotted,black,mark = x,every mark/.append style={solid,scale=0.8,fill=black}}
    }
\pgfplotscreateplotcyclelist{MyCol1b}{%
    {red},
	{dashed,blue},
    {dotted,black}
    }

\pgfplotscreateplotcyclelist{MyColors}{%
    {red,mark = *,every mark/.append style={solid,scale=0.4,fill=red}},
	{dotted,red,mark = *,every mark/.append style={solid,scale=0.4,fill=red}},
    {teal,mark = x,every mark/.append style={solid,scale=0.4,fill=teal}},
    {blue,mark = square*,every mark/.append style={solid,scale=0.4,fill=blue}},
    {dotted,teal,mark = x,every mark/.append style={solid,scale=0.4,fill=teal}},
    {dotted,blue,mark = square*,every mark/.append style={solid,scale=0.4,fill=blue}},
{dash dot,red,mark = *,every mark/.append style={solid,scale=0.4,fill=red}},
    {dash dot,teal,mark = x,every mark/.append style={solid,scale=0.4,fill=teal}},
    {dash dot,blue,mark = square*,every mark/.append style={solid,scale=0.4,fill=blue}},}

\tikzset{
  cross/.pic = {
    \draw[rotate = 45] (-#1,0) -- (#1,0);
    \draw[rotate = 45] (0,-#1) -- (0, #1);
  }
}

\providecommand{\ds}{\, \mathrm{d}s}
\providecommand{\du}{\, \mathrm{d}u}
\providecommand{\dx}{\, \mathrm{d}x}

\providecommand{\tria}{\mathcal{T}}

\providecommand{\CR}{\textup{CR}^1(\mathcal{T})}
\providecommand{\CRo}{\textup{CR}_0^1(\mathcal{T})}

\begin{document} 
	
\title[Guaranteed bounds and Ka{\v c}anov schemes via discrete duality]{Guaranteed upper bounds for iteration errors and modified Ka{\v c}anov schemes via discrete duality}

\author[L.\ Diening]{Lars Diening}
\author[J.\ Storn]{Johannes Storn}

\address[L.\ Diening]{Department of Mathematics, Bielefeld University, Postfach 10 01 31, 33501 Bielefeld, Germany}
\email{lars.diening@uni-bielefeld.de}
\address[J.\ Storn]{Faculty of Mathematics \& Computer Science, Institute of Mathematics, Leipzig University, Augustusplatz 10, 04109 Leipzig, Germany}
\email{johannes.storn@uni-leipzig.de}




\begin{abstract}
We apply duality theory to discretized convex minimization problems to obtain computable guaranteed upper bounds for the distance of given discrete functions and the exact discrete minimizer. Furthermore, we show that the discrete duality framework extends convergence results for the \Kacanov{} scheme to a broader class of problems.
\end{abstract}

\thanks{This work was funded by the Deutsche Forschungsgemeinschaft (DFG, German Research Foundation) – SFB 1283/2 2021 – 317210226.}

\maketitle


\section{Introduction}
\label{sec:introduction}
A wide class of non-linear partial differential equations arises from the minimization of convex energies $\mathcal{J}$ over some space $V$.
More precisely, given a convex integrand $\phi\colon [0,\infty) \to \mathbb{R}$, a bounded Lipschitz domain $\Omega\subset \mathbb{R}^d$, and a right-hand side $f\in V^*$, the goal is to find the minimizer
\begin{equation}\label{eq:Energy}
u = \argmin_{v\in V}\mathcal{J}(v)
\quad \text{with}\quad \mathcal{J}(v) \coloneqq \int_\Omega \phi(|\nabla v|) \dx - \int_\Omega fv\dx,
\end{equation}
where we replace the integral $\int_\Omega fv\dx$ by the duality pairing $\langle f,v\rangle_{V^*,V}$ for rough functionals $f\in V^*$.
Typical examples include shape-optimization problems or non-Newtonian fluid models, where the gradient is replaced by the symmetric gradient. An equivalent characterization of the minimizer seeks the solution $u\in V$ to the variational problem
\begin{equation}\label{eq:VarProb}
\int_\Omega \frac{\phi'(|\nabla u|)}{|\nabla u|} \nabla u \cdot \nabla v\dx = \int_\Omega f v\dx \qquad\text{for all }v\in V.
\end{equation}
To approximate the solution to these problems, the finite element method discretizes the space $V$ by some finite-dimensional space $V_h$. However, in general the discretized problems remain non-linear, which requires an iterative scheme such as Newton's method or a fixed-point iteration.
This paper addresses two key challenges arising from such computations:
\begin{itemize}
\item We discuss the computation of a guaranteed upper bound for the iteration error. Specifically, given a function $u_{h,n} \in V_h$, we derive a computable bound $\textup{GUB} < \infty$ such that
\begin{equation}\label{eq:IterationError}
\mathcal{J}(u_{h,n}) - \min_{v_h \in V_h} \mathcal{J}(v_h) \leq \textup{GUB}.
\end{equation}
\item We introduce a modified version of the \Kacanov{} iteration, an iterative scheme for computing the minimizer, which ensures convergence for problems where the classical \Kacanov{} scheme fails.
\end{itemize}
We achieve both results by an application of duality theory on the discrete level, an idea that was introduced in \cite{BalciDieningStorn23} for the $p$-Dirichlet energy but extends to a much larger class of problems.
This idea results in a dual minimization problem, where the dual energy is minimized over a constrained space $\Sigma$. Unlike the continuous case, the constraint is tested exclusively with functions in $V_h$, which allows us to compute functions in $\Sigma$ without requiring additional properties of the right-hand side $f$. 
Suitable dual functions can be obtained through post-processing. Alternatively, they emerge as a byproduct of the iterative scheme such as gradient descent, Newton, or \Kacanov{} schemes, leading to a built-in error estimator for the iteration error \eqref{eq:IterationError}.

Instead of solving the discretized primal problem, it is 
possible to solve the corresponding discrete dual problem. 
Since the convergence of iterative schemes often depends on specific properties of the integrand, the dual problem may provide advantages, allowing for example to obtain a convergent iterative scheme. This approach is particularly advantageous for the \Kacanov{} scheme, as for low-order discretizations, the dual iterative scheme simplifies to a linearization of the primal problem.

Our theoretical results are supported by a series of numerical experiments on the $p$-Laplace problem, the $p$-Stokes problem, a shape-optimization problem, and a Bingham fluid.

The remainder of this paper is organized as follows. Section~\ref{sec:Duality} introduces the duality theory. We apply this theory to discretized problems in Section~\ref{sec:DiscDuality}. In Section~\ref{sec:GUB}, we use the discrete duality result to obtain the guaranteed upper bound in \eqref{eq:IterationError} for the iteration error. Section~\ref{sec:dualKacanov} introduces the dual \Kacanov{} iteration. Our numerical experiments are displayed in Section~\ref{sec:numExp}.
\begin{remark}[Previous works]\label{rem:Han}
As noted by one referee, discrete duality was already employed in \cite[Sec.~5.4]{Han05} to derive guaranteed upper bounds for the iteration error in \Kacanov{} schemes, cf.~Section~\ref{sec:GUB} below. Another noteworthy application of discrete duality in \cite{Han05} is the comparison between the solution $u \in V$ to \eqref{eq:VarProb} and the solution $u_0 \in H^1_0(\Omega)$ of the linearized problem $-\Delta u_0 = f$, in the case of integrands satisfying $\phi'(t)/t \approx 1$ for all $t \geq 0$; see also \cite{Han94}.
\end{remark}
\section{Duality}\label{sec:Duality}
In this section we briefly discuss the well-known duality theory, see for example \cite{Zeidler90}. 
For the sake of presentation we consider energies defined as in \eqref{eq:Energy} with integrand $\phi$ being an N-function defined below. The arguments extend, however, to a larger class of problems.
\begin{definition}[Nice functions]\label{def:NfunctionIntro}
A function $\varphi\colon \mathbb{R}_{\geq 0} \to \mathbb{R}_{\geq 0}$ is an N-function if
\begin{enumerate}
\item $\varphi$ is continuous and convex,
\item there is a right-continuous and non-decreasing function $\varphi'\colon \mathbb{R}_{\geq 0} \to \mathbb{R}_{\geq 0}$ with $\varphi(t) = \int_0^t \varphi'(s)\ds$ that satisfies $\lim_{t\to \infty} \varphi'(t) = \infty$ as well as $\varphi'(0) = 0$ and $\varphi'(t) >0$ for all $t>0$.
\end{enumerate}
\end{definition}
The convex conjugate $\phi^*$ of the N-function $\phi$ reads for all $r\geq 0$
\begin{equation*}
\phi^*(r) \coloneqq \sup_{s\geq 0} \big(rs - \phi(s)\big). 
\end{equation*} 
Figure~\ref{fig:NfunctionsIntro} illustrates the convex conjugate $\varphi^*$ and indicates the following properties proven for example in  \cite{DieningEttwein08,DieningFornasierTomasiWank20}. 
\begin{figure}
  \begin{tikzpicture}[scale=1.6]
    \fill [gray!90, domain=0:2.6, variable=\x]
      (0, 0)
      -- plot ({\x}, {1-cos(deg(pi/3*\x)))})
      -- (2.6, 0)
      -- cycle;
      \fill [gray!30, domain=0:2, variable=\x]
      (0, 0)
      -- plot ({\x}, {1-cos(deg(pi/3*\x)))})
      -- (0, {1-cos(deg(pi/3*2))})
      -- cycle;

	\node[fill=white] () at (2,.7) {$\varphi(s)$};
	\node[fill=white] () at (.8,1) {$\varphi^*(r)$};
	\node[below] () at (2.6,0) {$s$};
	\node[left] () at (0,{1-cos(deg(pi/3*2))}) {$r$};

	\draw[dashed] (0,0) 
	-- (2.6,0)
	-- (2.6, {1-cos(deg(pi/3*2))}) 
	-- (0, {1-cos(deg(pi/3*2))});


%
	
    \draw [thick] [->] (0,0)--(2.8,0) node[right] {$t$};

    \draw [thick] [->] (0,0)--(0,2.1) node[above] {$u$};

    \draw [domain=0:2.8, variable=\x]
      plot ({\x}, {1-cos(deg(pi/3*\x)))}) node[right] {$\varphi'(t),\ (\varphi^*)'(u)$};
  \end{tikzpicture}
  \caption{Plot of derivatives $\varphi'$ and $(\varphi^*)'$. The area marked dark gray equals $\varphi(s)$, the area marked light gray equals $\varphi^*(r)$. The area surrounded by the dashed line equals $rs$.}\label{fig:NfunctionsIntro}
\end{figure}
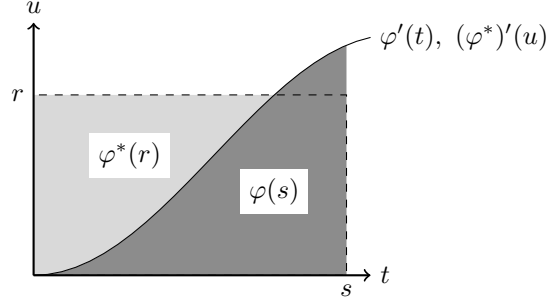%
\begin{proposition}[Conjugate of an N-function]\label{prop:Conjugate}
Let $\varphi\colon \mathbb{R}_{\geq 0} \to \mathbb{R}_{\geq 0}$ be an N-function and define the right-continuous inverse
\begin{equation*}
(\varphi')^{-1}(t) \coloneqq \sup\lbrace r \geq 0 \mid \varphi'(r) \leq t\rbrace\qquad\text{for all }t\geq 0.
\end{equation*}
Then we have the following.
\begin{enumerate}
\item If $\varphi'$ is strictly increasing, the function $(\varphi')^{-1}$ equals the inverse of $\varphi'$.
\item The convex conjugate of $\varphi$ is an $N$-function with the representation
\begin{equation*}
\varphi^*(r) =  \int_0^r (\varphi')^{-1}(u)\du\qquad\text{for all }r \geq 0.
\end{equation*} 
\item One has for all $t\geq 0$ the identity \label{itm:ConjugateC}
\begin{equation*}
\phi^*(\phi'(t)) = \phi'(t)t - \phi(t).
\end{equation*}
\item The function $\phi$ equals its bi-conjugate, that is for all $t\geq 0$\label{itm:biConjugate}
\begin{equation*}
\phi(t) = \phi^{**}(t) \coloneqq \sup_{s \geq 0} \big(t s - \phi^*(s)\big).
\end{equation*}
\end{enumerate}
 \end{proposition}
The results displayed in Proposition~\ref{prop:Conjugate} yield the following duality theorem, which we state for clarity of presentation for N-functions $\varphi$.
Notice that the duality theorem can be generalized to certain integrands $\phi$ such as the one in \eqref{eq:PhiBingham} that are not N-functions but just Young functions. We use the notation for scalar valued functions in $V$ and assume homogeneous Dirichlet boundary conditions. The arguments naturally extend to inhomogeneous mixed boundary conditions and vector-valued function spaces $V$, where the gradient $\nabla$ can be replaced by the symmetric gradient, see for example Section~\ref{subsec:pStokes}--\ref{subsec:Bingham} below.

We define the dual energy 
\begin{equation*}
\mathcal{J}^*(\tau) \coloneqq \int_\Omega \phi^*(|\tau|) \dx\qquad\text{for all }\tau \in L^1(\Omega;\mathbb{R}^d).
\end{equation*}
Moreover, we say that $\tau \in L^1(\Omega;\mathbb{R}^d)$ satisfies $-\divergence \tau = f$ in $V^*$, if
\begin{align}
\int_\Omega \tau \cdot \nabla v \dx = \int_\Omega f v \dx\qquad\text{for all }v\in V. 
\end{align}

\begin{theorem}[Duality]\label{thm:Duality}
For the energy $\mathcal{J}$ in \eqref{eq:Energy} with integrand $\phi$ being an N-function one has 
\begin{equation}\label{eq:PrimalDual}
\inf_{v\in V} \mathcal{J}(v) = -  \inf_{\substack{\tau \in L^1(\Omega;\mathbb{R}^d)\\ -\divergence \tau = f \text{ in }V^*}} \mathcal{J}^*(\tau).
\end{equation}
There exist minimizers $u\in V$ and $\sigma \in L^1(\Omega;\mathbb{R}^d)$ with $-\divergence \sigma = f$ in $V^*$; that is,
\begin{equation*}
u = \argmin_{v\in V} \mathcal{J}(v)\qquad\text{and}\qquad \sigma = \argmin_{\substack{\tau \in L^1(\Omega;\mathbb{R}^d)\\ -\divergence \tau = f \text{ in }V^*}} \mathcal{J}^*(\tau).
\end{equation*}
Minimizers are related via the identities
\begin{equation*}
\sigma = \frac{\phi'(|\nabla u|)}{|\nabla u|} \nabla u\qquad\text{and}\qquad \nabla u = \frac{(\phi^*)'(|\sigma|)}{|\sigma|} \sigma.
\end{equation*}
\end{theorem}
\begin{proof}
The identity in Proposition~\ref{prop:Conjugate}~\ref{itm:biConjugate}, together with the fact that exchanging $\inf$ and $\sup$ decreases the value of the expression, yields
\begin{subequations}
\begin{equation}\label{eq:ProofDual}
\begin{aligned}
\inf_{v\in V} \mathcal{J}(v) &= \inf_{v\in V} \int_\Omega \varphi(|\nabla v|) - fv \dx \\
& = \inf_{v\in V}\, \sup_{\tau\colon \Omega \to \mathbb{R}^d} \int_\Omega  \tau \cdot \nabla v - \phi^*(|\tau|) - fv \dx\\
&\geq  \inf_{v\in V}\, \sup_{\tau \in L^1(\Omega;\mathbb{R}^d)} \int_\Omega \tau \cdot \nabla v - \phi^*(|\tau|) - fv \dx \\
&  \geq \sup_{\tau \in L^1(\Omega;\mathbb{R}^d)} \inf_{v\in V} \int_\Omega \tau \cdot \nabla v - fv \dx - \int_\Omega  \phi^*(|\tau|) \dx \eqqcolon (*).
\end{aligned}
\end{equation}
Let $\tau \in L^1(\Omega;\mathbb{R}^d)$ be such that $\int_\Omega \tau \cdot \nabla v - fv \dx \neq 0$ for some $v \in V$. By appropriately scaling $v$, we see that $\inf_{v \in V} \int_\Omega \tau \cdot \nabla v - fv \dx = -\infty$. Since $(*)$ involves taking the supremum over all $\tau \in L^1(\Omega;\mathbb{R}^d)$, we may disregard functions $\tau$ for which the inner infimum is $-\infty$. That is, it suffices to consider the supremum over all $\tau \in L^1(\Omega;\mathbb{R}^d)$ satisfying $-\divergence \tau = f$ in $V^*$. This leads to the identity
\begin{equation}
(*) = \sup_{\substack{\tau \in L^1(\Omega;\mathbb{R}^d)\\ -\divergence \tau = f \text{ in }V^*}} -\int_\Omega \phi^*(|\tau|) \dx = - \inf_{\substack{\tau \in L^1(\Omega;\mathbb{R}^d)\\ -\divergence \tau = f \text{ in }V^*}} \mathcal{J}^*(\tau).
\end{equation}
\end{subequations}
Hence, the minimal energy in \eqref{eq:Energy} is bounded from below by the negative of the minimal dual energy. Note that both minima are attained, which can be established via the direct method in the calculus of variations, and they are unique due to the strict convexity of the energies.
Let $u\in V$ denote the solution to \eqref{eq:Energy} and define the function
\begin{equation}\label{eq:CharSig}
\sigma \coloneqq \frac{\phi'(|\nabla u|)}{|\nabla u|} \nabla u.
\end{equation}
Taking the absolute value shows $|\sigma | = \phi'(|\nabla u|)$. Using this identity in \eqref{eq:CharSig} yields
\begin{equation*}
\frac{|\nabla u|}{|\sigma|} \sigma = \nabla u.
\end{equation*}
Proposition~\ref{prop:Conjugate} implies that 
$|\nabla u| = (\varphi^*)'(\varphi'(|\nabla u|)) = (\varphi^*)'(|\sigma|)$. Combining theses results yields the characterization of $\nabla u$ in terms of $\sigma$, that is, 
\begin{equation*}
\nabla u = \frac{(\phi^*)'(|\sigma|)}{|\sigma|} \sigma.
\end{equation*}

The characterization of $u\in V$ in \eqref{eq:VarProb} reveals by testing with $v = u$ that
\begin{equation*}
\int_\Omega \phi'(|\nabla u|) |\nabla u| \dx = \int_\Omega fu\dx < \infty.
\end{equation*}
With the identity $|\sigma | = \phi'(|\nabla u|)$ and the monotonicity of $\phi'$ this bound leads to
\begin{equation*}
\int_\Omega |\sigma| \dx = \int_\Omega \phi'(|\nabla u|) \dx \leq \int_{\lbrace|\nabla u| \leq 1\rbrace} \phi'(1) \dx + \int_{\lbrace |\nabla u| > 1\rbrace}  \phi'(|\nabla u|) |\nabla u| \dx  < \infty.
\end{equation*}
Hence, $\sigma$ is in $L^1(\Omega;\mathbb{R}^d)$ and satisfies $-\divergence \sigma = f$ in $V^*$  due to \eqref{eq:VarProb}.
This property, the identity $|\sigma | = \phi'(|\nabla u|)$, and Proposition~\ref{prop:Conjugate}~\ref{itm:ConjugateC} yield
\begin{align*}
- \inf_{\substack{\tau \in L^1(\Omega;\mathbb{R}^d)\\ -\divergence \tau = f \text{ in }V^*}} \mathcal{J}^*(\tau) &\geq 
- \int_\Omega \varphi^*(|\sigma |)\dx  = - \int_\Omega \frac{\phi'(|\nabla u|)}{|\nabla u|} \nabla u\, \cdot  \nabla u - \varphi(|\nabla u|) \dx \\
&= \int_\Omega \varphi(|\nabla u|) - fu \dx  = \mathcal{J}(u) = \inf_{v\in V} \mathcal{J}(v).
\end{align*}
Combing this estimate with \eqref{eq:ProofDual} concludes the proof.
\end{proof}
It is a well-known technique, see e.g.~\cite{BraessSchoeberl08,Repin08,CarstensenMerdon14,ErnVohralik15,BartelsMilicevic20,BartelsKaltenbach23,FevotteRappaportVohralik24}, 
to compute a function $\tau \in L^1(\Omega;\mathbb{R}^d)$ with $- \divergence \tau = f$ in $V^*$ and use the duality relation in Theorem~\ref{thm:Duality} to obtain the guaranteed and computable upper bound 
\begin{equation*}
\mathcal{J}(v) - \min_{w\in V} \mathcal{J}(w) \leq \mathcal{J}(v) + \mathcal{J}^*(\tau)\qquad\text{for all }v\in V.
\end{equation*}
Our idea, used already for the $p$-Laplacian in \cite{BalciDieningStorn23}, is to use duality on the discrete level. 

\section{Discrete duality}\label{sec:DiscDuality}

In this section we apply the arguments from Section~\ref{sec:Duality} to the discretized minimization problem. 
Let $\tria$ be a regular triangulation of $\Omega \subset \mathbb{R}^d$, let $\mathbb{P}_k(T)$ denote the space of polynomials on $T$ with maximal degree $k\in \mathbb{N}_0$, and define 
\begin{equation*}
\begin{aligned}
\mathbb{P}_k(\tria) & \coloneqq \lbrace p_h \in L^1(\Omega) \colon p_h|_T \in \mathbb{P}_k(T)\text{ for all }T\in \tria\rbrace.
\end{aligned}
\end{equation*}
Our finite element space $V_h$ is a finite dimensional space such as
\begin{itemize}
\item the Lagrange finite element space $\mathcal{L}^1_{k,0}(\tria) \coloneqq  \mathbb{P}_k(\tria) \cap W^{1,1}_0(\Omega)$ of order $k\in \mathbb{N}$,
\item the Crouzeix--Raviart finite element space $\CRo \coloneqq \lbrace v_h \in \mathbb{P}_1(\tria)\colon v_h$ is continuous in all midpoints of interior faces of $\tria$ and equals zero in all barycenters of faces of $\tria$ on the boundary $\partial \Omega\rbrace$,
\item the space of discretely divergence free finite element functions such as the Taylor--Hood element
\begin{equation*}
\Big\lbrace v_h \in \mathcal{L}^1_{2,0}(\tria)^d  \colon \int_\Omega p_h \divergence v_h \dx = 0\text{ for all }p_h\in \mathcal{L}^1_1(\tria)\Big\rbrace,
\end{equation*}
the Crouzeix--Raviart element
\begin{equation*}
\Big\lbrace v_h \in \CRo^d  \colon \sum_{T\in \tria} \int_T p_h \divergence v_h \dx = 0\text{ for all }p_h \in \mathbb{P}_0(\tria)\Big\rbrace,
\end{equation*} 
or the Kouhia--Stenberg element (for $d=2$)
\begin{equation*}
\Big\lbrace v_h \in \mathcal{L}^1_{1,0}(\tria) \times \CRo \colon \sum_{T\in \tria} \int_T p_h \, \divergence v_h\dx = 0\text{ for all }p_h\in \mathbb{P}_0(\tria) \Big\rbrace.
\end{equation*}
\end{itemize}
Notice that the gradient is not necessarily defined as an $L^2$ function for certain non-conforming spaces such as $\CRo$. Thus, we replace the gradient by its broken version $
\nabla_h \colon W^{1,1}(\Omega) + \mathbb{P}_1(\tria) \to L^2(\Omega;\mathbb{R}^d)$ with
\begin{equation*}
(\nabla_h v)|_T \coloneqq \nabla v|_T\qquad \text{for all }v\in W^{1,1}(\Omega) + \mathbb{P}_1(\tria) \text{ and }T\in \tria.
\end{equation*}
Notice that $\nabla_h v = \nabla v$ for all functions $v\in W^{1,1}(\Omega)$. The resulting discretized problem seeks the minimizer
\begin{equation}\label{eq:discrePrimalProb}
u_h = \argmin_{v_h \in V_h} \mathcal{J}(v_h) \quad \text{with }\mathcal{J}(v_h) \coloneqq \int_\Omega \phi(|\nabla_h v_h|) \dx - \int_\Omega fv_h\dx.
\end{equation}
The minimizer is characterized as the unique solution to the variational problem: Seek $u_h \in V_h$ with 
\begin{equation}\label{eq:DiscrVarProb}
\int_\Omega \frac{\varphi'(|\nabla_h u_h|)}{|\nabla_h u_h|} \nabla_h u_h \cdot \nabla_h v_h \dx = \int_\Omega fv_h \dx\qquad\text{for all }v_h\in V_h.
\end{equation}
The corresponding dual problem involves the side constraint $-\divergence_h \tau = f$ in $V_h^*$, which is interpreted for all $\tau \in L^1(\Omega;\mathbb{R}^d)$ as 
\begin{equation}\label{eq:DiscrDiv}
\int_\Omega \tau \cdot \nabla_h v_h \dx = \int_\Omega fv_h\dx \qquad\text{for all }v_h \in V_h.
\end{equation}
The dual problem seeks the minimal value 
\begin{equation*}
\inf_{\substack{\tau \in L^1(\Omega;\mathbb{R}^d)\\ -\divergence_h \tau = f \text{ in }V_h^*}}  \mathcal{J}^*(\tau) \qquad\text{with }\mathcal{J}^*(\tau) \coloneqq  \int_\Omega \phi^*(|\tau|) \dx. 
\end{equation*} 
\begin{theorem}[Discrete duality]\label{thm:discreteDuality}
For the energy $\mathcal{J}$ in \eqref{eq:discrePrimalProb} with N-function $\varphi$ and discrete space $V_h$ one has 
\begin{equation*}
\min_{v_h\in V_h} \mathcal{J}(v_h) = -  \min_{\substack{\tau \in L^1(\Omega;\mathbb{R}^d)\\ -\divergence_h \tau = f \text{ in }V_h^*}} \mathcal{J}^*(\tau).
\end{equation*}
There exist minimizers $u_h\in V_h$ and $\sigma_h \in L^1(\Omega;\mathbb{R}^d)$ to the corresponding minimization problems; that is,
\begin{equation*}
\begin{aligned}
u_h = \argmin_{v_h\in V_h} \mathcal{J}(v_h)\qquad\text{and}\qquad \sigma_h = \argmin_{\substack{\tau \in L^1(\Omega;\mathbb{R}^d)\\ -\divergence_h \tau = f \text{ in }V_h^*}} \mathcal{J}^*(\tau).
\end{aligned}
\end{equation*}
Minimizers are related via the identity
\begin{equation*}
\sigma_h \coloneqq \frac{\phi'(|\nabla_h u_h|)}{|\nabla_h u_h|} \nabla_h u_h\qquad\text{and}\qquad \nabla_h u_h \coloneqq \frac{(\phi^*)'(|\sigma_h|)}{|\sigma_h|} \sigma_h.
\end{equation*}
\end{theorem}
\begin{proof}
The theorem follows by the same arguments as Theorem~\ref{thm:Duality}.
\end{proof}

\begin{corollary}[Constant gradients]\label{cor:Duality}
If $\nabla_h V_h \subset \mathbb{P}_0(\tria)^d$, one has   
\begin{equation}\label{eq:DiscretDualityP0}
\min_{v_h\in V_h} \mathcal{J}(v_h) = -  \min_{\substack{\tau \in \mathbb{P}_0(\tria)^d\\ -\divergence_h \tau = f \text{ in }V_h^*}} \mathcal{J}^*(\tau).
\end{equation}
\end{corollary}
\begin{proof}
Let $\Pi_0\colon L^1(\Omega;\mathbb{R}^d) \to \mathbb{P}_0(\tria)^d$ denote the $L^2(\Omega)$ orthogonal projection onto piece-wise constant functions. 
Any $\tau \in L^1(\Omega;\mathbb{R}^d)$ with $-\divergence_h \tau = f \text{ in }V_h^*$ satisfies
\begin{equation*}
\int_\Omega \Pi_0 \tau\cdot \nabla_h v_h \dx =  \int_\Omega \tau\cdot \nabla_h v_h \dx = \int_\Omega fv_h\dx \qquad\text{for all }v_h \in V_h.
\end{equation*}
Hence, the function $\Pi_0 \tau$ satisfies $-\divergence_h \Pi_0 \tau = f$ in $V_h^*$. Furthermore, Jensen's inequality reveals for the convex energy $\mathcal{J}^*$ that
\begin{equation*}
\mathcal{J}^*(\Pi_0 \tau) \leq \mathcal{J}^*(\tau).
\end{equation*}
Combining this observation with Theorem~\ref{thm:discreteDuality} concludes the proof.
\end{proof}

\begin{remark}[Marini identity]\label{rem:Marini}
It is known, cf.~\cite{Marini85} for the linear and \cite{CarstensenLiu15,LiuLiChen18,Bartels21} for the non-linear case, that for right-hand sides $f \in \mathbb{P}_0(\tria)$ the solution to the Crouzeix--Raviart FEM is linked to a discrete dual problem with lowest-order Raviart--Thomas elements $RT_0(\tria)$ in the sense that with $L^2(\Omega)$ orthogonal projection $\Pi_0 \colon L^1(\Omega;\mathbb{R}^d) \to \mathbb{P}_0(\tria)^d$ onto piece-wise constant functions one has \cite[Prop.~3.1]{Bartels21}
\begin{equation*}
\inf_{v_h \in \CRo} \mathcal{J}(v_h) = - \inf_{\substack{\tau_h \in RT_0(\tria)\\-\divergence_h \tau = f\text{ in }\CRo^*}} \mathcal{J}^*(\Pi_0\tau_h).
\end{equation*}
This result follows from the non-trivial observation that the discrete dual minimizer
\begin{equation*}
\sigma_h = \frac{\phi'(|\nabla_h u_h|)}{|\nabla_h u_h|} \nabla_h u_h \in \mathbb{P}_0(\tria)^d
\end{equation*}
in Corollary~\ref{cor:Duality} is the $L^2(\Omega)$ orthogonal projection of a Raviart--Thomas function $\tau_h \in RT_0(\tria)$ with $-\divergence_h \tau_h = f$ in $\CRo^*$ onto piece-wise constants; that is, $\sigma_h = \Pi_0 \tau_h$.
\end{remark}

\section{Guaranteed upper bound}\label{sec:GUB}
Due to the highly singular behavior of solutions to non-linear problems, adaptive finite element methods for non-linear variational problems as in \eqref{eq:Energy} are indispensable. 
Modern adaptive schemes, such as those discussed in \cite{ErnVohralik13,DieningFornasierTomasiWank20,BalciDieningStorn23,FevotteRappaportVohralik24} with current approximation $v_h \in V_h$ of the exact continuous minimizer $u \in V$ in \eqref{eq:Energy} consider the following error components:
\begin{enumerate}
    \item the discretization error $\mathcal{J}(v_h) - \mathcal{J}(u)$,\label{itm:DiscrError}
    \item the regularization error $|\mathcal{J}(v_h) - \mathcal{J}_\varepsilon(v_h)|$, if the solver employs a regularized energy $\mathcal{J}_\varepsilon$, and
    \item the iteration error $\mathcal{J}(v_h) -\min_{V_h} \mathcal{J}$.\label{itm:IterError}
\end{enumerate}
Based on these errors, adaptive routines perform one of the following actions:
\begin{enumerate}
    \item refine the underlying triangulation $\tria$,
    \item adjust the regularization parameter in $\mathcal{J}_\varepsilon$,
    \item execute an iteration in the solving routine for the discretized non-linear problem.
\end{enumerate}
Since in practical applications the exact error quantities are typically unknown, adaptive strategies rely on appropriate error estimators. However, these estimators often involve unknown constants that need to be determined experimentally to obtain (pre-asymptotically, cf.~\cite{HaberlPraetoriusSchimankoVohralik21,BeckerBrunnerInnerbergerMelenkPraetorius23}) good rates of convergence. 
This limitation is partially addressed by the following result, which provides a guaranteed upper bound for the iteration error without any unknown constants.
\begin{theorem}[Discrete primal-dual estimator]\label{thm:discrPrimalDual}
Let $u_h \in V_h$ denote the discrete minimizer in \eqref{eq:discrePrimalProb}, where the integrand $\phi$ is an N-function. 
Moreover, let $\tau \in L^1(\Omega;\mathbb{R}^d)$ with $-\divergence_h \tau = f$ in $V_h^*$ be given. Then one has the computable guaranteed upper bound
\begin{equation*}
\mathcal{J}(v_h) - \mathcal{J}(u_h) \leq \mathcal{J}(v_h) + \mathcal{J}^*(\tau) \qquad\text{for all }v_h \in V_h.
\end{equation*}
\end{theorem}
\begin{proof}
The theorem follows from the discrete duality result in Theorem~\ref{thm:discreteDuality}.
\end{proof}
\begin{remark}[Regularization error]\label{rem:AdaptReg}
As illustrated in the numerical experiment of Section~\ref{subsec:Bingham}, the discrete primal dual error estimator can additionally be used to estimate the impact of regularizations. In particular, suppose that $\mathcal{J}_\varepsilon$ is some regularization of the energy $\mathcal{J}$ and we have function $\tau \in L^1(\Omega;\mathbb{R}^d)$ with $-\divergence_h \tau = f$ in $V_h^*$ and $v_h \in V_h$. Discrete duality leads to the two upper bounds
\begin{equation*}
\begin{aligned}
\mathcal{J}_\varepsilon(v_h) - \min_{V_h} \mathcal{J}_\varepsilon &\leq \mathcal{J}_\varepsilon(v_h) + \mathcal{J}_\varepsilon^*(\tau)\eqqcolon \textup{GUB}_\varepsilon,\\
 \mathcal{J}(v_h) - \min_{V_h} \mathcal{J} &\leq \mathcal{J}(v_h) + \mathcal{J}^*(\tau) \eqqcolon \textup{GUB}.
\end{aligned}
\end{equation*}
A very large ratio $\textup{GUB}/\textup{GUB}_\varepsilon$  indicates a significant impact of the regularization. Hence, adjusting the regularization instead of proceeding to solve the regularized problems can be beneficial.
\end{remark}
A necessary requirement for the guaranteed upper bound in Theorem~\ref{thm:discrPrimalDual} is a function $\tau \in L^1(\Omega;\mathbb{R}^d)$ with $-\divergence_h \tau = f$ in $V_h^*$. For certain iterative schemes such functions result as a byproduct of the method. 

For example, given a function $u_{h,n} \in V_h$, the \Kacanov{} scheme computes iteratively $u_{h,n+1} \in V_h$ with
\begin{equation}\label{eq:DefKacanov}
\int_\Omega \frac{\phi'(|\nabla_h u_{h,n}|)}{|\nabla_h u_{h,n}|} \nabla_h u_{h,n+1} \cdot \nabla_h v_h \dx = \int_\Omega fv_h\dx\qquad\text{for all }v_h \in V_h.
\end{equation}
We derive a function $\sigma_{n+1} \in L^1(\Omega; \mathbb{R}^d)$ with $-\divergence_h \sigma_{n+1} = f$ in $V_h^*$ by defining
\begin{equation}\label{eq:FctKacanov} 
\sigma_{n+1} \coloneqq \frac{\phi'(|\nabla_h u_{h,n}|)}{|\nabla_h u_{h,n}|} \nabla_h u_{h,n+1}. 
\end{equation}

Similarly, gradient descent methods such as in \cite{HuangLiLiu07} seek the descent direction $g_h \in V_h$ with 
\begin{equation*}
\int_\Omega \nabla_h g_h \cdot \nabla_h v_h \dx = \int_\Omega \frac{\phi'(|\nabla_h u_{h,n}|)}{|\nabla_h u_{h,n}|} \nabla_h u_{h,n} \cdot \nabla_h v_h \dx - \int_\Omega f v_h \dx\quad\text{for all } v_h \in V_h.
\end{equation*} 
This leads to a function $\sigma_{n+1} \in L^1(\Omega;\mathbb{R}^d)$ with $-\divergence_h \sigma_{n+1} = f$ in $V_h^*$ defined by 
\begin{equation*}
\sigma_{n+1} \coloneqq \frac{\phi'(|\nabla_h u_{h,n}|)}{|\nabla_h u_{h,n}|} \nabla_h u_{h,n} - \nabla_h g_h.
\end{equation*}

Furthermore, Newton schemes such as in \cite{LuoTeng16} seek $\delta_n \in V_h$ with
\begin{equation*}
\begin{aligned}
  \lefteqn{\int_\Omega \bigg(\phi''(|\nabla_h u_{h,n}|)-\frac{\phi'(\abs{\nabla_h u_{h,n}})}{\abs{\nabla_h u_{h,n}}}\bigg)\bigg(\frac{\nabla_h u_{h,n}}{|\nabla_h u_{h,n}|} \cdot \nabla_h \delta_n\bigg)\frac{\nabla_h u_{h,n}}{|\nabla_h u_{h,n}|} \cdot \nabla_h v_h\dx}
  \\
  &+ \int_\Omega \frac{\phi'(|\nabla_h u_{h,n}|)}{|\nabla_h u_{h,n}|} \nabla_h \delta_n \cdot \nabla_h v_h\dx\\
&\qquad \qquad = - \int_\Omega \frac{\phi'(|\nabla_h u_{h,n}|)}{|\nabla_h u_{h,n}|} \nabla_h u_{h,n} \cdot \nabla_h v_h\dx + \int_\Omega f v_h\dx\qquad\text{for all }v_h \in V_h.
\end{aligned}
\end{equation*}
This leads to a function $\sigma_{n+1} \in L^1(\Omega;\mathbb{R}^d)$ with $-\divergence_h \sigma_{n+1} = f$ in $V_h^*$ defined by 
\begin{equation*}
\begin{aligned}
\sigma_{n+1}&\coloneqq\bigg(\phi''(|\nabla_h u_{h,n}|)-\frac{\phi'(\abs{\nabla_h u_{h,n}})}{\abs{\nabla_h u_{h,n}}}\bigg)\bigg(\frac{\nabla_h u_{h,n}}{|\nabla_h u_{h,n}|} \cdot \nabla_h \delta_n\bigg)\frac{\nabla_h u_{h,n}}{|\nabla_h u_{h,n}|}\\
&\quad + \frac{\phi'(|\nabla_h u_{h,n}|)}{|\nabla_h u_{h,n}|} \nabla_h (u_{h,n}+\delta_n).
\end{aligned}
\end{equation*}
Thus, all of these iterative schemes have the built-in iteration error estimator
\begin{equation*}
\mathcal{J}(u_{h,n+1}) - \mathcal{J}(u_h) \leq \mathcal{J}(u_{h,n+1}) + \mathcal{J}^*(\sigma_{n+1}).
\end{equation*}
This property has already been used in the adaptive scheme in \cite{BalciDieningStorn23} but extends to a much larger class of problems and solvers. Remark~\ref{rem:DiscEquili} below discusses the derivation of functions $\sigma_{n+1}$ if we exploit inexact solvers for the linearized problems.
\begin{remark}[Equilibration]
Let $\mathcal{L}^1_{1,0}(\tria) \subset V_h$.
The equilibration techniques introduced in \cite{BraessSchoeberl08} and generalized in \cite{ErnVohralik15} allow us to compute for any $\tau \in L^1(\Omega;\mathbb{R}^d)$ with $-\divergence_h \tau = f$ in $V_h^*$  functions $\mathcal{E}\tau \in \mathbb{P}_{k+1}(\tria;\mathbb{R}^d)$ with divergence $-\divergence \mathcal{E}\tau = f$ for right-hand sides $f\in \mathbb{P}_k(\tria)$. These functions can be used to compute a guaranteed upper bound for the distance to the exact continuous minimal energy; that is, 
\begin{equation*}
\mathcal{J}(v) - \min_{w \in V} \mathcal{J}(w) \leq \mathcal{J}(v)  + \mathcal{J}^*(\mathcal{E}\tau)\qquad\text{for all }v \in V.
\end{equation*}
\end{remark}
\begin{remark}[Discrete equilibration]\label{rem:DiscEquili}
The paper \cite{PapezRuedeVohralikWohlmuth20} introduces a multi-level approach to design a function  $\sigma_h \in L^2(\Omega;\mathbb{R}^d)$ with $-\divergence \chi_h = r_h$ for some piecewise polynomial $r_h \in \mathbb{P}_k(\tria)$ of maximal degree $k\in \mathbb{N}$ satisfying for a given weight $A\colon \Omega \to \mathbb{R}$ and a discrete function $v_h \in V_h$ 
\begin{align*}
\int_\Omega r_h  w_h\dx = \int_\Omega A \nabla_h v_h\cdot \nabla_h w_h \dx - \int_\Omega fw_h \dx \qquad\text{for all }w_h \in V_h.
\end{align*}
It is used in \cite[Thm.~7.1]{PapezRuedeVohralikWohlmuth20} to bound the iteration error in linear solvers. More precisely, they show that for $A = I$ the distance to $u_h \in V_h$ with 
\begin{align}\label{eq:Laplace}
\int_\Omega \nabla_h u_h\cdot \nabla_h w_h\dx = \int_\Omega f w_h\dx\qquad\text{for all }w_h \in V_h
\end{align}
 is bounded by 
\begin{align}\label{eq:asdsadssadasdsada}
\begin{aligned}
\lVert \nabla_h (u_h - v_h)\rVert_{L^2(\Omega)} &= \sup_{w_h \in V_h\setminus \lbrace 0 \rbrace} \frac{\int_\Omega \nabla_h (u_h - v_h)\cdot \nabla_h w_h\dx}{\lVert \nabla_h w_h \rVert_{L^2(\Omega)}}  \\
 &  = \sup_{w_h \in V_h\setminus \lbrace 0 \rbrace} \frac{\int_\Omega \chi_h\cdot \nabla_h w_h\dx}{\lVert \nabla_h w_h \rVert_{L^2(\Omega)}}
 \leq \lVert \chi_h \rVert_{L^2(\Omega)}.
\end{aligned}
\end{align}
Since the function $\sigma_h \coloneqq A \nabla_h v_h - \chi_h$ satisfies $-\divergence \sigma_h = f$ in $V_h^*$, we can exploit Theorem~\ref{thm:discrPrimalDual} to obtain even in the nonlinear case the guaranteed upper bound
\begin{align*}
\mathcal{J}(v_h) - \min_{w_h \in V_h} \mathcal{J}(w_h) \leq \mathcal{J}(v_h) + \mathcal{J}^*(\sigma_h).
\end{align*}
A straightforward calculation shows that for the Laplace problem in \eqref{eq:Laplace} we recover the same bound as in \eqref{eq:asdsadssadasdsada} in the sense that 
\begin{align*}
\frac{1}{2} \lVert \nabla_h (u_h - v_h) \rVert^2_{L^2(\Omega)} = \mathcal{J}(v_h) - \mathcal{J}(u_h) \leq 
\mathcal{J}(v_h) + \mathcal{J}^*(\sigma_h) = \frac{1}{2}\lVert \chi_h \rVert_{L^2(\Omega)}^2. 
\end{align*}
\end{remark}
\section{Dual \Kacanov{} scheme}\label{sec:dualKacanov}
Another important application of the discrete duality result in Theorem~\ref{thm:discrPrimalDual} are dual \Kacanov{} schemes introduced for the $p$-Laplacian with $p > 2$ in \cite{BalciDieningStorn23}.
This ansatz is motivated by the fact the convergence results for the \Kacanov{} scheme  in \cite{Zeidler90,HanJensenShimansky97,GarauMorinZuppa11,HeidWihler20} require among others the property 
\begin{equation*}
 \phi''(t) \leq \frac{\phi'(t)}{t}	\qquad \text{for all }t\geq 0.
\end{equation*} 
In particular, these results do not apply to integrands with 
\begin{equation}\label{eq:OtherCase}
 \phi''(t) \geq \frac{\phi'(t)}{t}	\qquad \text{for all }t\geq 0.
\end{equation}
However, assuming that $\phi'$ is strictly increasing and differentiable, we apply Proposition~\ref{prop:Conjugate} to conclude $(\phi^*)' = (\phi')^{-1}$. This yields 
\begin{equation*}
1 = \frac{d}{ds} s = \frac{d}{ds}  \phi'\big((\phi^*)'(s)\big) = \phi''\big((\phi^*)'(s)\big) \, (\phi^*)''(s)
\end{equation*}
and we obtain from \eqref{eq:OtherCase} for $t = (\phi^*)'(s)$ with $s \geq 0$ the inequality
\begin{equation*}
\phi''((\phi^*)'(s)) \geq \frac{s}{(\phi^*)'(s)}\qquad \text{implying} \qquad (\phi^*)''(s) \leq \frac{(\phi^*)'(s)}{s}.
\end{equation*} 
Hence, the dual function satisfies the required property, which extends existing convergence results to the case \eqref{eq:OtherCase} if one considers the dual problem.

To discretize the dual problem, we suggest the following. 
Suppose that $V_h \subset \mathbb{P}_1(\tria)$ consists of piece-wise affine functions such as lowest-order Lagrange or Crouzeix--Raviart elements.
Corollary~\ref{cor:Duality} states the discrete duality relation
\begin{equation*}
\min_{v_h\in V_h} \mathcal{J}(v_h) = -  \min_{\substack{\tau \in \mathbb{P}_0(\tria)^d\\ -\divergence_h \tau = f \text{ in }V_h^*}} \mathcal{J}^*(\tau).
\end{equation*}
Hence, the dual \Kacanov{} scheme computes instead of the minimizer to the discretized primal problem the minimizer of the discrete dual problem via the following iteration: Given $\sigma_{h,n} \in \mathbb{P}_0(\tria)^d$, compute $\sigma_{h,n+1} \in  \mathbb{P}_0(\tria)^d$ and $u_{h,n+1} \in V_h$ that satisfy for all $\tau_h \in \mathbb{P}_0(\tria)^d$ and $v_h\in V_h$
\begin{equation}\label{eq:DualKacSaddlePoint}
\begin{aligned}
\int_\Omega \frac{(\phi^*)'(|\sigma_{h,n}|)}{|\sigma_{h,n}|} \sigma_{h,n+1} \cdot \tau_h \dx + \int_\Omega \nabla_h u_{h,n+1} \cdot \tau_h\dx &= 0 ,\\
\int_\Omega \sigma_{h,n+1} \cdot \nabla_h v_h\dx &= \int_\Omega fv_h\dx.
\end{aligned}
\end{equation}
The first line in the saddle point problem \eqref{eq:DualKacSaddlePoint} yields the point-wise identity 
\begin{equation*}
\frac{(\phi^*)'(|\sigma_{h,n}|)}{|\sigma_{h,n}|} \sigma_{h,n+1} = \nabla_h u_{h,n+1}.
\end{equation*}
Inserting this identity into the second equation of \eqref{eq:DualKacSaddlePoint} yields the following equivalent formulation of the dual \Kacanov{} scheme.
\begin{definition}[Dual \Kacanov{} iteration]\label{def:DualKacanov}
Given $\sigma_{h,n}\in \mathbb{P}_0(\tria)^d$, seek $u_{h,n+1}\in V_h$ with 
\begin{equation*}
\int_\Omega \frac{|\sigma_{h,n}|}{(\phi^*)'(|\sigma_{h,n}|)} \nabla_h u_{h,n+1} \cdot \nabla_h v_h\dx = \int_\Omega fv_h\dx\qquad\text{for all }v_h \in V_h.
\end{equation*}
Then update the weight 
\begin{equation}\label{eq:weightDualKac}
\sigma_{h,n+1} \coloneqq \frac{|\sigma_{h,n}|}{(\phi^*)'(|\sigma_{h,n}|)} \nabla_h u_{h,n+1}.
\end{equation}
\end{definition}
\begin{remark}[Higher-order]
The idea of the dual \Kacanov{} scheme extends to higher-order discretizations such as $V_h = \mathcal{L}^1_{2,0}(\tria)$. However, for general integrands $\phi$ and $\nabla_h V_h \not\subset \mathbb{P}_0(\tria)^d$ the weight $\sigma_{h,n+1}$ in \eqref{eq:weightDualKac} is not a piece-wise polynomial function, leading to practical difficulties.
\end{remark}

\begin{remark}[Gradient flow]
An alternative to \Kacanov{} schemes is the use of gradient flows. Bartels suggests in \cite[Sec.5]{Bartels21} the application of duality and the Marini identity (see Remark~\ref{rem:Marini}) to extend convergence results from \cite{BartelsDieningNochetto18} for the gradient flow of the $p$-Laplacian to the previously uncovered case $p>2$ for the Crouzeix--Raviart FEM.
The discrete duality result shows that this idea is not restricted to the Crouzeix--Raviart FEM but can also be extended to (lowest-order) Lagrange elements.
\end{remark}
\section{Numerical experiments}\label{sec:numExp}
In this section, we illustrate the performance of the discrete primal-dual estimator in Theorem~\ref{thm:discrPrimalDual} and the dual \Kacanov{} scheme in Definition~\ref{def:DualKacanov} numerically.
\subsection{$p$-Laplace problem}\label{subsec:pLap}
In our first experiment we solve the $p$-Laplace problem on the L-shaped domain $\Omega = (-1,1)^2\setminus [0,1)^2$ with homogeneous Dirichlet boundary condition. We discretize the space $W^{1,p}_0(\Omega)$ by lowest-order Lagrange and Crouzeix--Raviart elements, cf.~Section~\ref{sec:DiscDuality}. The underlying triangulation $\tria$ consists of $\#\tria \approx 400\, 000$ triangles and resulted from the AFEM loop for the Laplace problem with constant right-hand side $f = 2$ discretized by lowest-order Lagrange elements, cf.~\cite{CarstensenFeischlPagePraetorius14}. Hence, the mesh is graded towards the re-entrant corner. The right-hand side $f = 2$ in \eqref{eq:Energy} equals two and the integrand reads with constants $\kappa \coloneqq 10^{-1}$ and $p\in (1,\infty)$
\begin{equation*}
\phi(t) \coloneqq \int_0^t \phi'(s)\ds \quad \text{ with }\quad \phi'(s) \coloneqq s\, (\kappa + s)^{p-2}.
\end{equation*}
Our initial iterates equal $u_{h,0} \coloneqq 0$ and $\sigma_{h,0} \coloneqq 0$. We solve the problem with the \Kacanov{} scheme \eqref{eq:DefKacanov} for $p \leq 2$ and the dual \Kacanov{} scheme in Definition~\ref{def:DualKacanov} for $p> 2$. The convergence of the first approach  has been studied for $p\leq 2$ in \cite{DieningFornasierTomasiWank20}, the convergence of the latter for $p\geq 2$ in \cite{BalciDieningStorn23}. Notice that the use of the dual \Kacanov{} scheme is indeed needed for the cases $p>2$, since the primal \Kacanov{} scheme does not converge, as indicated by our numerical experiments (not displayed in this paper) and discussed for $\kappa = 0$ in~\cite[Rem.~21]{DieningFornasierTomasiWank20}. 

Since in our computations the energies obtained by the Crouzeix--Raviart FEM were very close to the ones obtained by the Lagrange FEM, we solely display the results of the latter in Figure~\ref{fig:Exp1}. Let $u_h \in \mathcal{L}^1_{1,0}(\tria)$ denote the exact discrete minimizer. The convergence history plot in Figure~\ref{fig:Exp1} indicates that for $p<2$ the primal and dual energy differences converge with the same speed, but the dual energy is significantly smaller. This yields an efficiency index 
\begin{equation}\label{eq:efficiencyIndex}
\frac{\mathcal{J}^*(\sigma_{h,n}) + \mathcal{J}(u_{h,n})}{\mathcal{J}(u_{h,n}) - \mathcal{J}(u_h)} = 1 + \frac{\mathcal{J}^*(\sigma_{h,n}) + \mathcal{J}(u_h)}{\mathcal{J}(u_{h,n}) - \mathcal{J}(u_h)}
\end{equation}
close to one. For $p>2$ the situation seems to be different. 
The dual energy error seems to converge with a slower rate than the primal energy error, leading to an increasing efficiency index. Nevertheless, both \Kacanov{} schemes result in fast linear convergence. 
\begin{figure}[ht!]
\begin{tikzpicture}
\begin{axis}[
clip=false,
width=.5\textwidth,
height=.45\textwidth,
ymode = log,
cycle multi list={\nextlist MyCol1},
scale = {1},
clip = true,
legend cell align=left,
legend style={legend columns=1,legend pos= south west,font=\fontsize{7}{5}\selectfont}
]
	\addplot table [x=Iter,y=EnergyError] {experiments/Experiment1_nrElem388485.dat};
	\addplot table [x=Iter,y=DualEnergyError] {experiments/Experiment1_nrElem388485.dat};
	\addplot table [x=Iter,y=EfficiencyIndex] {experiments/Experiment1_nrElem388485.dat};
	\legend{{$\mathcal{J}(u_{h,n}) - \mathcal{J}(u_h)$},{$\mathcal{J}^*(\sigma_{h,n}) + \mathcal{J}(u_h)$},{Efficiency index}};
\end{axis}
\end{tikzpicture}
\begin{tikzpicture}
\begin{axis}[
clip=false,
width=.5\textwidth,
height=.45\textwidth,
ymode = log,
ymax = 50,
cycle multi list={\nextlist MyCol1},
scale = {1},
clip = true,
legend cell align=left,
legend style={legend columns=1,legend pos= south west,font=\fontsize{7}{5}\selectfont}
]
	\addplot table [x=Iter,y=EnergyError] {experiments/Experiment1b_nrElem388485.dat};
	\addplot table [x=Iter,y=DualEnergyError] {experiments/Experiment1b_nrElem388485.dat};
	\addplot table [x=Iter,y=EfficiencyIndex] {experiments/Experiment1b_nrElem388485.dat};
	\legend{{$\mathcal{J}(u_{h,n}) - \mathcal{J}(u_h)$},{$\mathcal{J}^*(\sigma_{h,n}) + \mathcal{J}(u_h)$},{Efficiency index}};
\end{axis}
\end{tikzpicture}
\caption{Convergence history of the energy differences and the efficiency index \eqref{eq:efficiencyIndex} plotted against the number of iterations $n$ for the $p$-Laplacian with $p=3/2$ (left) and $p=4$ (right).} \label{fig:Exp1}
\end{figure}
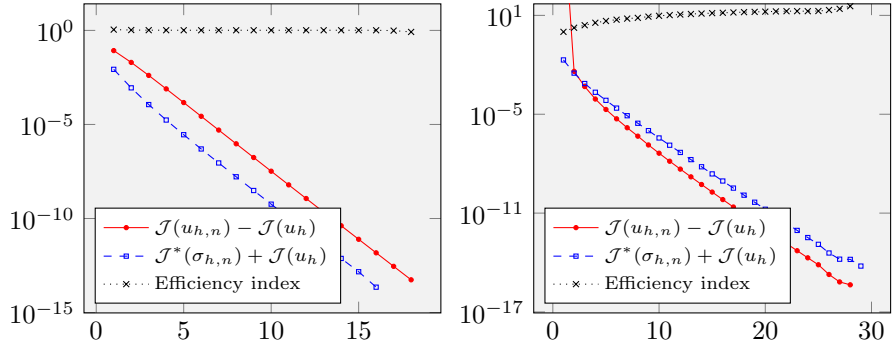%

%
%
%

\subsection{Optimal design problem}\label{subsec:ShapeOpti}
Our second experiment solves the optimal design problem for maximal torsion stiffness of an infinite bar of given geometry and unknown distribution of two materials of prescribed amounts. The model and its numerical approximation is investigated for example in~\cite{GoodmanKohnReyna86,KawohlStaraWittum91,BartelsCarstensen08,CarstensenTran21}. With parameters
\begin{equation*}
\lambda \coloneqq 0.0145,\quad \mu_1 \coloneqq 1,\quad \mu_2 \coloneqq 2,\quad t_1 \coloneqq \sqrt{2\lambda\mu_1/\mu_2},\quad t_2 \coloneqq \mu_2\lambda/\mu_1 
\end{equation*}
the derivative of the integrand $\phi$ in \eqref{eq:Energy} reads for all $s\geq 0$
\begin{equation*}
\phi'(s) = \begin{cases}
\mu_2 s&\text{for }s\leq t_1,\\
\mu_2t_1 = t_2 \mu_1 &\text{for }t_1 < s \leq t_2,\\
\mu_1 s &\text{for }t_2 < s.
\end{cases}
\end{equation*}
We have a constant right-hand side $f=1$ and we include homogeneous Dirichlet boundary conditions. We use the same domain and graded mesh as in the experiment of Section~\ref{subsec:pLap}. The problem is solved by the \Kacanov{} scheme in \eqref{eq:DefKacanov} with initial data $u_{h,0} = 0$. The exact discrete minimizer is denoted by $u_h \in V_h$.

Again, the results for the Lagrange and Crouzeix--Raviart FEM are similar and thus we display in Figure~\ref{fig:Exp2} solely the results of the Lagrange FEM. The convergence history plot in Figure~\ref{fig:Exp2} illustrates a fast convergence of the \Kacanov{} scheme in a pre-asymptotic regime. After this pre-asymptotic phase, the iterative scheme seems to converge linearly with some moderate speed. Thereby, the dual energy converges slower than the primal energy. This causes the growth of the efficiency index.

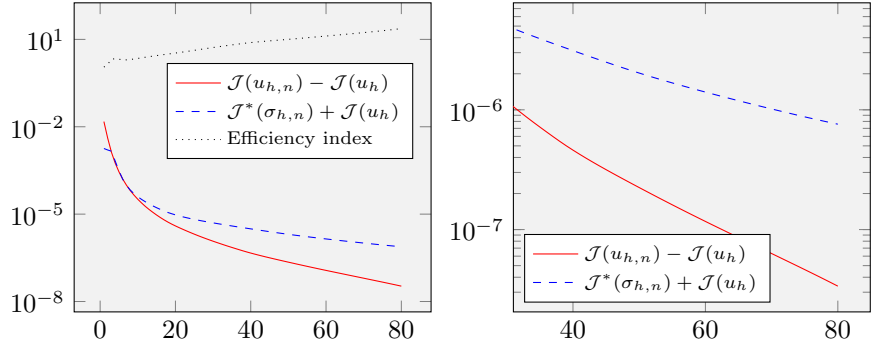
\begin{figure}[ht!]
\begin{tikzpicture}
\begin{axis}[
clip=false,
width=.5\textwidth,
height=.45\textwidth,
ymode = log,
cycle multi list={\nextlist MyCol1b},
scale = {1},
clip = true,
legend cell align=left,
legend style={legend columns=1,at={(.95,0.65)},anchor = east,font=\fontsize{7}{5}\selectfont}
]
	\addplot table [x=Iter,y=EnergyError] {experiments/Experiment4_nrElem388485.dat};
	\addplot table [x=Iter,y=DualEnergyError] {experiments/Experiment4_nrElem388485.dat};
	\addplot table [x=Iter,y=EfficiencyIndex] {experiments/Experiment4_nrElem388485.dat};	
	\legend{{$\mathcal{J}(u_{h,n}) - \mathcal{J}(u_h)$},{$\mathcal{J}^*(\sigma_{h,n}) + \mathcal{J}(u_h)$},{Efficiency index}};
\end{axis}
\end{tikzpicture}
\begin{tikzpicture}
\begin{axis}[
clip=false,
width=.5\textwidth,
height=.45\textwidth,
ymode = log,
cycle multi list={\nextlist MyCol1b},
scale = {1},
xmin = 31,
clip = true,
legend cell align=left,
legend style={legend columns=1,legend pos= south west,font=\fontsize{7}{5}\selectfont}
]
	\addplot table [x=Iter,y=EnergyError] {experiments/Experiment4_nrElem388485.dat};
	\addplot table [x=Iter,y=DualEnergyError] {experiments/Experiment4_nrElem388485.dat};
	\legend{{$\mathcal{J}(u_{h,n}) - \mathcal{J}(u_h)$},{$\mathcal{J}^*(\sigma_{h,n}) + \mathcal{J}(u_h)$}};
\end{axis}
\end{tikzpicture}
\caption{Convergence history of the energy differences and the efficiency index \eqref{eq:efficiencyIndex} plotted against the number of iterations $n$ for the optimal design problem with (left) and without (right) iterations from 1 to 30.} \label{fig:Exp2}
\end{figure}%

\subsection{$p$-Stokes problem}\label{subsec:pStokes}
The $p$-Stokes problem involves the integrand $\phi$ from Section~\ref{subsec:pLap} and minimizes with symmetric gradient $\varepsilon(\bigcdot) \coloneqq 1/2\, (\nabla \bigcdot + (\nabla \bigcdot)^\top)$ the energy
\begin{equation*}
\mathcal{J}(v) \coloneqq \int_\Omega \phi\big(|\varepsilon(v+ v_\Gamma)|\big) - fv\dx
\end{equation*} 
over all $v\in Z \coloneqq \lbrace w \in W^{1,p}_0(\Omega;\mathbb{R}^d)\colon \divergence w = 0\rbrace$, where the right-hand side $f=0$ is set to zero and the function $v_\Gamma \in W^{1,p}(\Omega;\mathbb{R}^d)$ satisfies $\divergence v_\Gamma = 0$. The domain reads $\Omega \coloneqq \big((-2, 8) \times  (-1, 1)) \setminus ([-2, 0] \times [-1, 0]\big)$ and the boundary values of $v_\Gamma$ are defined for all $(x,y) \in \partial \Omega$ as
\begin{equation}\label{eq:BddCondStokes}
v_\Gamma(x,y) \coloneqq \begin{cases}
 \big((1-y)y/10,0\big)&\text{for }x=-2,\\
\big(0,0\big)&\text{for }-2<x<8,\\
\big((1+y)(1-y)/80,0\big)  &\text{for }x = 8.
\end{cases}
\end{equation} 
Let $A : B$ denote the Frobenius inner product for matrices $A,B \in \mathbb{R}^{d\times d}$.
Following the arguments of Theorem~\ref{thm:Duality}, the dual problem minimizes  the energy
\begin{equation*}
\mathcal{J}^*(\tau) \coloneqq \int_\Omega \varphi^*(|\tau|) - \tau : \varepsilon (v_\Gamma) \dx
\end{equation*}
over all $\tau \in L^1(\Omega;\mathbb{R}^{d\times d})$ with $ \int_\Omega \tau:\varepsilon(v)\dx = 0$ for all $v \in Z$.
We discretize the space $Z$ by the Kouhia--Stenberg finite element, see \cite{KouhiaStenberg95,CarstensenSchedensack15,HuSchedensack19}, that is, 
\begin{equation*}
Z_h \coloneqq \Big\lbrace v_h \in \mathcal{L}^1_{1,0}(\tria) \times \CRo \colon \sum_{T\in \tria} \int_T p_h \, \divergence v_h\dx = 0\text{ for all }p_h\in \mathbb{P}_0(\tria) \Big\rbrace. 
\end{equation*} 
Moreover, we replace the symmetric gradient by the broken symmetric gradient $\varepsilon_h(\bigcdot) \coloneqq 1/2\, (\nabla_h \bigcdot + (\nabla_h \bigcdot)^\top)$ and approximate the function $v_\Gamma$ by the discrete function $v_{\Gamma,h} \in \mathcal{L}^1_{1}(\tria) \times \CR$ which equals zero at its interior degrees of freedom and $v_\Gamma$ at its degrees in freedom on the boundary $\partial\Omega$.
The resulting discretized non-linear problem is solved by the \Kacanov{} scheme \eqref{eq:DefKacanov} for $p\leq 2$ and by the dual \Kacanov{} scheme in Definition~\ref{def:DualKacanov} for $p>2$ with initial values $u_{h,0} \coloneqq v_{\Gamma,h}$ and $\sigma_{h,0} \coloneqq 0$. The triangulation $\tria$ consists of $\# \tria \approx 100\, 000$ triangles and is graded towards the re-entrant corner.

Figure~\ref{fig:Exp3} displays the resulting convergence history and the efficiency index \eqref{eq:efficiencyIndex}. 
The observations are similar to the $p$-Laplace problem: For $p<2$ the dual energy is much smaller than the primal energy error, leading to an efficiency index close to one. For $p>2$ the dual energy error is larger than the primal energy error, leading to an efficiency index close to 10. In contrast to the $p$-Laplace problem, the convergence speed of the primal and dual energy appears to be equal.

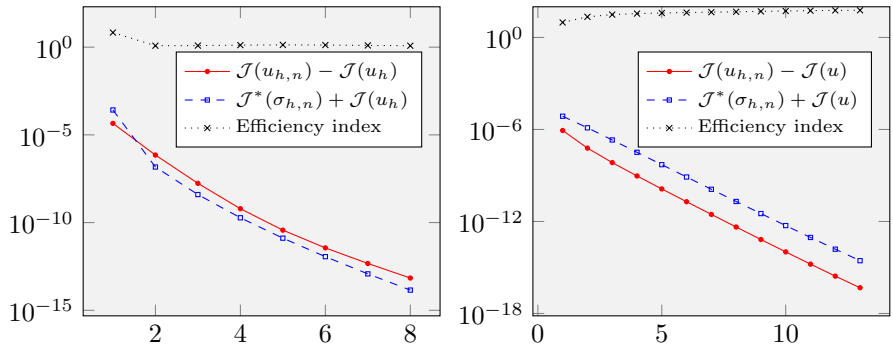
\begin{figure}[ht!]
\begin{tikzpicture}
\begin{axis}[
clip=false,
width=.5\textwidth,
height=.45\textwidth,
ymode = log,
cycle multi list={\nextlist MyCol1},
scale = {1},
clip = true,
legend cell align=left,
legend style={legend columns=1,at={(.95,0.7)},anchor = east,font=\fontsize{7}{5}\selectfont}
]
	\addplot table [x=Iter,y=EnergyError] {experiments/Experiment2_nrElem92890.dat};
	\addplot table [x=Iter,y=DualEnergyError] {experiments/Experiment2_nrElem92890.dat};
	\addplot table [x=Iter,y=EfficiencyIndex] {experiments/Experiment2_nrElem92890.dat};
	\legend{{$\mathcal{J}(u_{h,n}) - \mathcal{J}(u_h)$},{$\mathcal{J}^*(\sigma_{h,n}) + \mathcal{J}(u_h)$},{Efficiency index}};
\end{axis}
\end{tikzpicture}
\begin{tikzpicture}
\begin{axis}[
clip=false,
width=.5\textwidth,
height=.45\textwidth,
ymode = log,
ymax = 100,
cycle multi list={\nextlist MyCol1},
scale = {1},
clip = true,
legend cell align=left,
legend style={legend columns=1,at={(.95,0.7)},anchor = east,font=\fontsize{7}{5}\selectfont}
]
	\addplot table [x=Iter,y=EnergyError] {experiments/Experiment2b_nrElem92890.dat};
	\addplot table [x=Iter,y=DualEnergyError] {experiments/Experiment2b_nrElem92890.dat};
	\addplot table [x=Iter,y=EfficiencyIndex] {experiments/Experiment2b_nrElem92890.dat};
	\legend{{$\mathcal{J}(u_{h,n}) - \mathcal{J}(u)$},{$\mathcal{J}^*(\sigma_{h,n}) + \mathcal{J}(u)$},{Efficiency index}};
\end{axis}
\end{tikzpicture}
\caption{Convergence history of the energy differences and the efficiency index \eqref{eq:efficiencyIndex} plotted against the number of iterations $n$ for the $p$-Stokes problem with $p=3/2$ (left) and $p=4$ (right).} \label{fig:Exp3}
\end{figure}%

\subsection{Bingham fluid} \label{subsec:Bingham}
Our last experiment investigates the discrete primal dual error estimator for slowly flowing Bingham fluids. The function spaces $Z$ and $Z_h$, the domain $\Omega$, the boundary data $v_\Gamma$, and the right-hand side $f=0$ are the same as in Section~\ref{subsec:pStokes}. The integrand reads 
\begin{equation}\label{eq:PhiBingham}
\phi(t) \coloneqq \nu t^2 +  \sigma t \qquad \text{for all }t\geq 0.
\end{equation}
We set the yield stress to $\sigma \coloneqq 0.3$ and the fluid viscosity to $\nu\coloneqq 1$. Moreover, we set the regularized energy
\begin{equation*}
\phi_\varepsilon(t) \coloneqq \nu t^2 +  \sigma (\sqrt{ t^2 + \varepsilon^{2}} - \varepsilon) \qquad \text{for all }t\geq 0\text{ and } \varepsilon > 0.
\end{equation*}
As suggested in \cite{HeidSueli22} we solve the discretized problem by the primal \Kacanov{} scheme \eqref{eq:DefKacanov}, where the integrand $\phi$ is replaced by the regularized integrand $\phi_{1/n}$ with $n \in \mathbb{N}$ corresponding to the number of the current iteration in the \Kacanov{} scheme. Alternatively, we decrease the regularization parameter adaptively according to the strategy discussed in Remark~\ref{rem:AdaptReg}. More precisely, starting with $\varepsilon = 1$ we divide $\varepsilon$ by two if we have a ratio 
\begin{equation*}
\frac{\textup{GUB}}{\textup{GUB}_\varepsilon} \coloneq \frac{\mathcal{J}(u_{h,n}) + \mathcal{J}^*(\sigma_{h,n})}{\mathcal{J}_\varepsilon(u_{h,n}) + \mathcal{J}_\varepsilon^*(\sigma_{h,n})} > 100.
\end{equation*}
The convergence history in Figure~\ref{fig:Exp4} demonstrates that the adaptive selection of the regularization parameter $\varepsilon$ leads to significantly faster convergence compared to the fixed decrease strategy $\varepsilon \coloneqq 1/n$, highlighting the clear advantage of the adaptive approach.
\begin{figure}[ht!]
\begin{tikzpicture}
\begin{axis}[
clip=false,
width=.5\textwidth,
height=.45\textwidth,
ymode = log,
cycle multi list={\nextlist MyCol1b},
scale = {1},
clip = true,
legend cell align=left,
legend style={legend columns=1,at={(.95,0.4)},anchor = east,font=\fontsize{7}{5}\selectfont}
]
	\addplot table [x=Iter,y=Energy0Error] {experiments/Experiment3_nrElem92890.dat};
	\addplot table [x=Iter,y=DualEnergy0Error] {experiments/Experiment3_nrElem92890.dat};
	\legend{{$\mathcal{J}(u_{h,n}) - \mathcal{J}(u_h)$},{$\mathcal{J}^*(\sigma_{h,n}) + \mathcal{J}(u_h)$}};
\end{axis}
\end{tikzpicture}
\begin{tikzpicture}
\begin{axis}[
clip=false,
width=.5\textwidth,
height=.45\textwidth,
ymode = log,
cycle multi list={\nextlist MyCol1b},
scale = {1},
clip = true,
legend cell align=left,
legend style={legend columns=1,legend pos= north east,font=\fontsize{7}{5}\selectfont}
]
	\addplot table [x=Iter,y=Energy0Error] {experiments/Experiment3_adapt_nrElem92890.dat};
	\addplot table [x=Iter,y=DualEnergy0Error] {experiments/Experiment3_adapt_nrElem92890.dat};
		\addplot table [x=Iter,y=eps] {experiments/Experiment3_adapt_nrElem92890.dat};
	\legend{{$\mathcal{J}(u_{h,n}) - \mathcal{J}(u_h)$},{$\mathcal{J}^*(\sigma_{h,n}) + \mathcal{J}(u_h)$},{Regularization $\varepsilon $}};
\end{axis}
\end{tikzpicture}
\caption{Convergence history of the energy differences plotted against the number of iterations $n$ with regularization $\varepsilon \coloneqq 1/n$ (left) and adaptively decreasing regularization (right).} \label{fig:Exp4}
\end{figure}
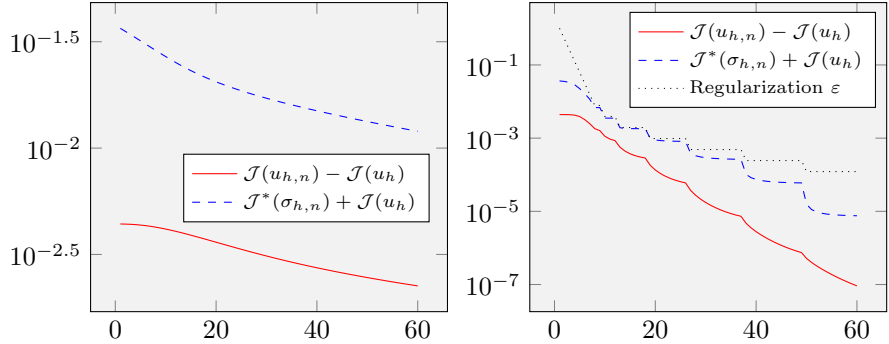%


\printbibliography
\end{document}